\documentclass[12pt]{article}
\usepackage{amssymb,amsthm,amsmath,latexsym}
\usepackage{url}
\newtheorem{thm}{Theorem}
\newtheorem{prop}[thm]{Proposition}
\newtheorem{lem}[thm]{Lemma}

\theoremstyle{remark}

\newcommand{\FF}{\mathbb{F}}

\DeclareMathOperator{\supp}{supp}

\begin{document}

\title{On the existence of $s$-extremal singly even self-dual codes}

\author{
Masaaki Harada\thanks{
Research Center for Pure and Applied Mathematics,
Graduate School of Information Sciences,
Tohoku University, Sendai 980--8579, Japan.
email: {\tt mharada@tohoku.ac.jp}.}
}

\maketitle

\begin{abstract}
We construct new $s$-extremal singly even self-dual codes 
with minimum weights $8,10,12$ and $14$.
We also give tables for
the currently known results on the existence of
$s$-extremal singly even self-dual codes
with minimum weights $8,10,12$ and $14$.
\end{abstract}

\section{Introduction}

Codes over $\FF_2$ are called {\em binary}, where
$\FF_2$ denotes the binary field.
All codes in this note are binary.
The {\em dual} code $C^{\perp}$ of a code 
$C$ of length $n$ is defined as
$
C^{\perp}=
\{x \in \FF_2^n \mid x \cdot y = 0 \text{ for all } y \in C\},
$
where $x \cdot y$ is the standard inner product.
A code $C$ is called 
{\em self-dual} if $C = C^{\perp}$. 
Self-dual codes are divided into two classes.
A self-dual code $C$ is {\em doubly even} if the weights of all
codewords of $C$ are multiples of $4$, and {\em
singly even} if there is at least one codeword of weight $\equiv 2
\pmod 4$.
It is  known that a self-dual code of length $n$ exists 
if and only if  $n$ is even, and
a doubly even self-dual code of length $n$
exists if and only if $n \equiv 0 \pmod 8$.
The minimum weight $d$ of a self-dual code of length $n$
is bounded by
$d  \le 4  \lfloor{n/24} \rfloor + 4$ 
if $n \not\equiv 22 \pmod {24}$ and
$d  \le 4 \lfloor{n/24} \rfloor + 6$ 
if $n \equiv 22 \pmod {24}$~\cite{MS73} and \cite{Rains}.
A self-dual code meeting the bound is called  {\em extremal}.
A self-dual code is called  {\em optimal} if it has the largest
minimum weight among all self-dual codes of that length.
Two codes are {\em equivalent} if one can be
obtained from the other by permuting the coordinates.

Let $C$ be a singly even self-dual code and let $C_0$ denote the
subcode of codewords of weight $\equiv0\pmod4$. Then $C_0$ is
a subcode of codimension $1$ \cite{C-S}. The {\em shadow} $S$ of $C$ is
defined to be $C_0^\perp \setminus C$~\cite{C-S}.
The minimum weight $d(D)$ of a coset $D$ is the minimum non-zero weight of
all vectors in $D$.
Let $C$ be a singly even 
self-dual code of length $n$ and let $S$ be the shadow of $C$.
It was shown in~\cite{BG} that 
\begin{equation}\label{eq:BG}
d(S) \le \frac{n}{2}+4-2d(C),
\end{equation}
unless $n \equiv 22 \pmod{24}$ and $d(C)= 4 \lfloor n/24 \rfloor +6$,
in such a case $d(S) = n/2+8-2d(C)$.
A singly even self-dual code meeting the bound
is called {\em $s$-extremal}~\cite{BG}.
It was shown that $s$-extremal singly even
self-dual codes imply
$1$-designs sometimes $2$-designs~\cite[Theorem~3.1]{BG}.
In addition, the existence of many $s$-extremal singly even self-dual codes was
mentioned in~\cite{BG}.
Some restrictions on lengths
for which there is an $s$-extremal singly even self-dual code
are known~\cite{G07}, \cite{HK08}, \cite{HM} and \cite{Kim}.


A classification of 
$s$-extremal singly even self-dual codes with
minimum weight $4$ was done in~\cite{Elkies}.
Up to equivalence, 
all $s$-extremal singly even self-dual codes with
minimum weight $6$ are known
(\cite{AGKSS}, \cite[p.~29]{BG}, \cite{B06} and~\cite{HM36}).
The previously known results on the existence of
$s$-extremal singly even self-dual codes can be found in~\cite{BG}
for minimum weights $8, 10, \ldots, 18$.
In this note, we explicitly construct new $s$-extremal singly even
self-dual codes with minimum weights $8,10,12$ and $14$.
All the new $s$-extremal singly even self-dual codes are
extremal or optimal, except for the codes of length $56$.
We also give tables for
the currently known results on the existence of
$s$-extremal singly even self-dual codes
with minimum weights $8,10,12$ and $14$.

This note is organized as follows.
In Section~\ref{sec:con}, we give methods for constructing
$s$-extremal singly even self-dual codes with parameters
$[24k+4,12k+2,4k+2]$,
$[24k+12,12k+6,4k+4]$ and
$[24k+20,12k+10,4k+4]$,
by considering neighbors.
In Section~\ref{sec:d8}, a great number of new 
$s$-extremal and extremal singly even self-dual codes with minimum
weight $8$ are explicitly constructed for lengths $42$ and $44$.
We note that the only lengths for which 
the classification has not been done yet are $42$ and $44$
for minimum weight $8$.
In Section~\ref{sec:d10}, 
new $s$-extremal singly even self-dual codes with minimum weight $10$
are explicitly constructed for length $56$.
In Section~\ref{sec:d12}, 
new $s$-extremal and extremal singly even self-dual codes with minimum weight $12$
are explicitly constructed for lengths $62$ and $64$.
In Section~\ref{sec:d14}, many new 
$s$-extremal and optimal singly even self-dual codes with minimum weight $14$
are explicitly constructed for length $78$.
Also, the current knowledge on the existence of 
$s$-extremal singly even self-dual codes with minimum weights $8,10,12$
and $14$ is presented in
Tables~\ref{Tab:d8}, \ref{Tab:d10}, \ref{Tab:d12}
and~\ref{Tab:d14}, respectively.
All computer calculations in Sections~\ref{sec:d8},
\ref{sec:d10}, \ref{sec:d12} and \ref{sec:d14}
were done with the help of {\sc Magma}~\cite{Magma}.

 \section{Construction of $s$-extremal singly even self-dual codes of
lengths $24k+4$, $24k+12$ and $24k+20$}\label{sec:con}

Two self-dual codes $C$ and $C'$ of length $n$
are said to be {\em neighbors} if $\dim(C \cap C')=n/2-1$. 
Let $C$ be a singly even self-dual code and let $C_0$ denote the
doubly even subcode, that is, the
subcode of codewords having weight $\equiv0\pmod4$. 
Since  $C_0$ is a subcode of codimenion $1$~\cite{C-S},
there are cosets
$C_1,C_2,C_3$ of $C_0$ such that $C_0^\perp = C_0 \cup C_1 \cup
C_2 \cup C_3 $, where $C = C_0  \cup C_2$ and $S = C_1 \cup C_3$.
If $C$ is a singly even
self-dual code of length $n \equiv 0 \pmod 4$, then 
$C_0 \cup C_1$ and $C_0 \cup C_3$ are self-dual (see~\cite{BP}).
It is trivial that $C_0 \cup C_1$ and $C_0 \cup C_3$ are
neighbors of $C$.
We denote these neighbors by $N_1(C)$ and $N_3(C)$, respectively.
Since
$C_0^\perp = C_0 \cup C_1 \cup C_2 \cup C_3 $ and $C_0$ is the
doubly even subcode for both $N_1(C)$ and $N_3(C)$, 
the shadows of $N_1(C)$ and $N_3(C)$ are
$C_2 \cup C_3$ and $C_2 \cup C_1$, respectively.
In addition, if $n \equiv 0 \pmod 8$, then
the two codes are doubly even self-dual (see~\cite{BP}).


\begin{prop}\label{prop:sec3}
\begin{enumerate}
\renewcommand{\labelenumi}{\rm (\roman{enumi})}
\item
Let $C$ be an $s$-extremal singly even self-dual $[24k+4,12k+2,4k+2]$ code.
Then the neighbors $N_1(C)$ and $N_3(C)$ of  $C$
are also $s$-extremal singly even self-dual $[24k+4,12k+2,4k+2]$ codes.
\item
Let $C$ be an $s$-extremal and extremal singly even self-dual $[24k+20,12k+10,4k+4]$ code.
Then the neighbors $N_1(C)$ and $N_3(C)$ of  $C$
are also $s$-extremal and extremal singly even self-dual $[24k+20,12k+10,4k+4]$ codes.
\item
Let $C$ be an $s$-extremal singly even self-dual $[24k+12,12k+6,4k+2]$ code.
Then the neighbors $N_1(C)$ and $N_3(C)$ of  $C$
are $s$-extremal and extremal singly even self-dual $[24k+12,12k+6,4k+4]$ codes.
\end{enumerate}
\end{prop}
\begin{proof}
\begin{enumerate}
\renewcommand{\labelenumi}{\rm (\roman{enumi})}
\item 
Since $C$ has minimum weight $4k+2$, 
$d(C_0) \ge 4k+4$ and $d(C_2) =4k+2$.
Since the shadow of $C$ has minimum weight $4k+2$,
$d(C_1) \ge 4k+2$ and $d(C_3) \ge 4k+2$.
Let $S_1$ and $S_3$ denote the shadows of $N_1(C)$ and $N_3(C)$,
respectively.
Then we have
\begin{equation}\label{eq:24k+4}
d(S_1)=d(S_3)=4k+2. 
\end{equation}

Suppose that $d(C_1) \ge 4k+6$.  
From the upper bounds on the minimum weights of self-dual codes,
it holds that $d(N_1(C)) = 4k+4$.  
From~\eqref{eq:BG}, we have
\[
d(S_1) \le \frac{24k+4}{2} + 4 - 2d(N_1(C)) = 4k-2.
\]
This contradicts~\eqref{eq:24k+4}.
Hence, $d(C_1) = 4k+2$ and $d(N_1(C))=4k+2$.
Similarly, $d(C_3) = 4k+2$ and $d(N_3(C))=4k+2$.

\item
Since $C$ has minimum weight $4k+4$, 
$d(C_0)=4k+4$ and $d(C_2) \ge 4k+6$.
Since the shadow of $C$ has minimum weight $4k+6$,
$d(C_1) \ge 4k+6$ and $d(C_3) \ge 4k+6$.
By a argument similar to the proof of (i),
$N_1(C)$ and $N_3(C)$ have minimum weight $4k+4$, and
the shadows of $N_1(C)$ and $N_3(C)$
have minimum weight $4k+6$.

\item
Since $C$ has minimum weight $4k+2$, 
$d(C_0) \ge 4k+4$ and $d(C_2) =4k+2$.
Since the shadow of $C$ has minimum weight $4k+6$,
$d(C_1) \ge 4k+6$ and $d(C_3) \ge 4k+6$.
By a argument similar to the proof of (i),
$N_1(C)$ and $N_3(C)$ have minimum weight $4k+4$, and
the shadows of $N_1(C)$ and $N_3(C)$ have minimum weight $4k+2$.
\end{enumerate}
This completes the proof.
\end{proof}

The above method is used to construct
$s$-extremal and extremal singly even self-dual $[44,22,8]$ codes and
$s$-extremal and optimal singly even self-dual $[76,38,14]$ codes.



 \section{Existence of $s$-extremal singly even self-dual codes with
minimum weight 8}\label{sec:d8}

In this section, a great number of new 
$s$-extremal and extremal singly even self-dual codes with minimum weight $8$
are constructed for lengths $42$ and $44$.
The current knowledge on the existence of 
$s$-extremal singly even self-dual codes with minimum weight
$8$ is also presented.


\subsection{Range of lengths}
If there is an $s$-extremal
singly even self-dual $[n,n/2,d]$
code with $d \equiv 0 \pmod 4$, then $n \le 6d-4$~\cite[Proposition~3.1]{Kim}.
Hence, if there is an $s$-extremal
singly even self-dual $[n,n/2,8]$ code, then $n \le 44$.
Also, if there is a singly even self-dual $[n,n/2,8]$ code,
then $n=32,36$ and $n \ge 38$ (see~\cite{C-S}).
Up to equivalence, 
for lengths $32,36,38$ and $40$,
all extremal singly even self-dual codes are known
in~\cite{C-S}, \cite{MG08}, \cite{AGKSS} and~\cite{BBH}, respectively.
We note that the only lengths for which 
the classification of  extremal singly even self-dual codes
has not been done yet are $42$ and $44$
for minimum weight $8$.

\subsection{Length 42}
Currently, 
$17$ inequivalent $s$-extremal and extremal
singly even self-dual $[42,21,8]$ codes
are known (see~\cite[p.~30]{BG}).
The code $R4$ in~\cite[Table~III]{C-S} is the first known code.

Any self-dual code of length $n$ can be reached
from any other by taking successive neighbors (see~\cite{C-S}).
It is known that a self-dual code $C$ of length $n$ has
$2(2^{n/2-1}-1)$ self-dual neighbors.
These neighbors are constructed by finding
$2^{n/2-1}-1$ subcodes of codimension $1$ in $C$
containing the allone vector.
A computer program written in {\sc Magma},
which was used to find self-dual neighbors,
can be obtained electronically from\\
\url{http://www.math.is.tohoku.ac.jp/~mharada/Paper/neighbor.txt}.
By finding all $2(2^{20}-1)$ self-dual neighbors of $R4$,
we determined the equivalence classes among
$s$-extremal and extremal self-dual $[42,21,8]$ neighbors of $R4$.
The number of inequivalent $s$-extremal and extremal 
singly even self-dual $[42,21,8]$ neighbors of $R4$, which is not
equivalent to $R4$, is $48232$.
Hence, we have the following:

\begin{prop}\label{prop:42}
There are at least $48233$
inequivalent $s$-extremal and extremal singly even self-dual $[42,21,8]$ codes.
\end{prop}

The $48232$ codes are constructed as
\[
N_{42,i}=
\langle (R4 \cap \langle x_i \rangle^\perp), x_i \rangle,
\]
where $x_i$ can be obtained from\\
\url{http://www.math.is.tohoku.ac.jp/~mharada/Paper/42-21-8-se.txt}.


\subsection{Length 44} 
There are five inequivalent pure double circulant
extremal singly even self-dual $[44,22,8]$
codes and these codes are denoted by $P_{44,i}$
$(i=1,2,3,4,5)$~\cite[Table~2]{HGK}.
Note that the codes $P_{44,1}$ and $P_{44,2}$ are $s$-extremal.
By Proposition~\ref{prop:sec3}, two $s$-extremal and extremal 
singly even self-dual $[44,22,8]$ neighbors 
$N_1(P_{44,i})$ and $N_3(P_{44,i})$
are constructed from $P_{44,i}$ $(i=1,2)$.
We verified that 
$N_3(P_{44,1})$ is equivalent to $P_{44,2}$,
$N_3(P_{44,2})$ is equivalent to $P_{44,1}$ and
$N_1(P_{44,1})$ is equivalent to $N_1(P_{44,2})$.
In addition, we verified that $N_1(P_{44,1})$ is not equivalent to
$P_{44,i}$ $(i=1,2)$.
Thus, $N_1(P_{44,1})$ is a new 
$s$-extremal and extremal singly even self-dual $[44,22,8]$ code.

Moreover, we found all $s$-extremal and extremal 
singly even self-dual $[44,22,8]$ neighbors of 
$P_{44,i}$ $(i=1,2,3,4,5)$.  
The numbers of inequivalent
$s$-extremal and extremal singly even self-dual $[44,22,8]$ neighbors of
$P_{44,i}$, which are inequivalent to 
$P_{44,1}$ and $P_{44,2}$ $(i=1,2,3,4,5)$, are
$99$, 
$99$, 
$0$, 
$0$ and  
$0$, 
respectively.
We denote the $99$ inequivalent
$s$-extremal and extremal singly even self-dual $[44,22,8]$ neighbors of
$P_{44,1}$ (resp.\ $P_{44,2}$) by
$N_{44,1,j}$ (resp.\ $N_{44,2,j}$) $(j=1,2,\ldots,99)$.
The $198$ codes are constructed as
\[
N_{44,i,j}=
\langle (P_{44,i} \cap \langle x_{i,j} \rangle^\perp), x_{i,j} \rangle,
\]
where $x_{1,j}$ can be obtained from\\
\url{http://www.math.is.tohoku.ac.jp/~mharada/Paper/44-22-8-se-1.txt}
and $x_{2,j}$ can be obtained from\\
\url{http://www.math.is.tohoku.ac.jp/~mharada/Paper/44-22-8-se-2.txt}.
We verified that 
$N_{44,1,99}$ and $N_{44,2,99}$ are equivalent, and 
there is no pair of equivalent codes among the other codes.
Hence, there are at least $199$
inequivalent $s$-extremal and extremal singly even self-dual $[44,22,8]$ codes.
%
%
We remark that
$N_1(P_{44,1})$ is equivalent to $N_{44,1,99}$ and
$N_1(P_{44,2})$ is equivalent to $N_{44,2,99}$.

Up to equivalence, 
there is a unique extremal singly even self-dual $[46,23,10]$
code~\cite{HMT}.
The code $C_{46}$ has generator matrix
$
\left(\begin{array}{ccccc}
I_{23}  &  R_{23}  \\
\end{array}\right)
$, where $R_{23}$ is the circulant matrix with first row
\[
(0,0,0,1,0,0,0,1,1,0,0,1,0,1,1,0,0,1,0,1,0,0,1),
\]
and $I_n$ denotes the identity matrix of order $n$.
We verified that all $s$-extremal and extremal singly even self-dual
$[44,22,8]$ codes, which are obtained from $C_{46}$ by
subtracting two coordinates $i,j$,
are divided into $29$ equivalence classes.
Here, since the automorphism group of $C_{46}$ acts 
transitively on the coordinates,
we may assume that $i=1$.
Let $S_{44,1,i}$ denote the $s$-extremal and extremal singly even self-dual
$[44,22,8]$ code which is obtained from $C_{46}$ by
subtracting two coordinates $1,i$.
The $29$ inequivalent codes are constructed as $S_{44,1,i}$, where
\begin{multline*}
i=
 2, 3, 4, 5, 6, 7, 8, 9,10,11,12,24,25,26,27,
\\
28,29,31,32,33,34,36,37,38,39,40,41,43,46.
\end{multline*}
We verified that there is no pair of
equivalent codes among the $199$ codes and the $29$ codes.
Hence, we have the following:

\begin{prop}\label{prop:44}
There are at least $228$
inequivalent $s$-extremal and extremal singly even self-dual $[44,22,8]$ codes.
\end{prop}

\subsection{Table $N_8(n)$}

We summarize in Table~\ref{Tab:d8} the number $N_{8}(n)$
of known inequivalent 
$s$-extremal singly even self-dual $[n,n/2,8]$ codes,
along with the references.

\begin{table}[thb]
\caption{Minimum weight $8$}
\label{Tab:d8}
\begin{center}
{\small
\begin{tabular}{c|c|c||c|c|c}
\noalign{\hrule height0.8pt}
$n$ & $N_{8}(n)$ & References &
$n$ & $N_{8}(n)$ & References \\
\hline
32 & $3$     & \cite{C-S} & 40 & $3597997 $ &\cite{BBH}\\
36 & $25$    & \cite{MG08} &42 & $\ge 48233$& Proposition~\ref{prop:42} \\
38 & $1730$  & \cite{AGKSS} &44 & $\ge 228$  & Proposition~\ref{prop:44} \\

 \noalign{\hrule height0.8pt}
\end{tabular}
}
\end{center}
\end{table}

\section{Existence of $s$-extremal singly even self-dual codes with
minimum weight 10}\label{sec:d10}

In this section, new
$s$-extremal singly even self-dual codes with minimum weight $10$
are constructed for length $56$.
Note that the largest minimum weight among known singly even
self-dual codes of length $56$ is $10$.
The current knowledge on the existence of 
$s$-extremal singly even self-dual codes with minimum weight
$10$ is also presented.


\subsection{Range of lengths}
If there is
an $s$-extremal  singly even self-dual $[n,n/2,10]$ code,
then $46 \le n \le 70$~\cite[Corollary~II.4]{HK08}.
For $n \in \{46,50,52,54,58\}$, the existence of 
$s$-extremal singly even self-dual
$[n,n/2,10]$ codes was mentioned in~\cite[p.~30]{BG}.
For $n \in \{48,50,52,54\}$, many 
$s$-extremal singly even self-dual $[n,n/2,10]$ codes are already known.


\subsection{Length 56}
Let $A$ and $B$ be the circulant matrices
with first rows
\[
(0,0,0,1,0,0,1,1,0,0,0,1,1,1) \text{ and }
(1,0,1,0,1,0,0,1,0,1,0,1,0,1),
\]
respectively.
Let $C_{56}$ be the $[56,28]$ code with generator matrix
\begin{equation} \label{eq:GM}
\left(
\begin{array}{ccc@{}c}
\quad & {\Large I_{28}} & \quad &
\begin{array}{cc}
A & B \\
B^T & A^T
\end{array}
\end{array}
\right),
\end{equation}
where $A^T$ denotes the transpose of a matrix $A$.
We verified that $C_{56}$ is an 
$s$-extremal singly even self-dual $[56,28,10]$ code.
In addition,
we found all inequivalent $s$-extremal
singly even self-dual $[56,28,10]$ neighbors of $C_{56}$.
The number of the neighbors, which is not
equivalent to $C_{56}$, is $20$.
The $20$ codes $N_{56,i}$ $(i=1,2,\ldots,20)$ are constructed as
\[
N_{56,i}=
\langle (C_{56} \cap \langle x_i \rangle^\perp), x_i \rangle,
\]
where the supports $\supp(x_i)$ of $x_i$ are listed in Table~\ref{Tab:56}.
Currently, $71$ inequivalent
$s$-extremal singly even self-dual $[56,28,10]$ codes are
known~\cite{HM}.
We verified that the known $71$ codes, $C_{56}$ and
$N_{56,i}$ $(i=1,2,\ldots,20)$ are inequivalent.
Hence, we have the following:

\begin{prop}\label{prop:56}
There are at least $92$
inequivalent $s$-extremal singly even self-dual $[56,28,10]$ codes.
\end{prop}

\begin{table}[thb]
\caption{New $s$-extremal singly even self-dual $[56,28,10]$ codes}
\label{Tab:56}
\begin{center}
{\small
\begin{tabular}{c|l}
\noalign{\hrule height0.8pt}
Codes & \multicolumn{1}{c}{$\supp(x_i)$} \\
\hline
$N_{56, 1}$&$\{2,3,10,31,34,39,44,53,55,56\}$\\
$N_{56, 2}$&$\{1,11,12,15,25,28,29,36,46,53\}$\\
$N_{56, 3}$&$\{11,22,26,34,36,40,43,44,46,47\}$\\
$N_{56, 4}$&$\{4,7,13,17,23,27,30,32,44,48\}$\\
$N_{56, 5}$&$\{10,23,33,38,44,45,46,48,49,53\}$\\
$N_{56, 6}$&$\{2,10,21,24,27,36,41,48,49,50\}$\\
$N_{56, 7}$&$\{1,4,8,33,39,42,46,50,52,55\}$\\
$N_{56, 8}$&$\{8,12,13,18,23,24,28,33,44,51\}$\\
$N_{56, 9}$&$\{2,11,12,14,16,18,23,51,53,54\}$\\
$N_{56,10}$&$\{19,22,27,30,37,38,41,43,54,55\}$\\
$N_{56,11}$&$\{9,13,15,16,23,26,29,35,42,48\}$\\
$N_{56,12}$&$\{3,9,11,13,17,20,23,29,35,50\}$\\
$N_{56,13}$&$\{5,7,13,23,32,34,36,39,42,44\}$\\
$N_{56,14}$&$\{11,13,14,17,23,25,26,31,36,49\}$\\
$N_{56,15}$&$\{3,10,13,17,31,37,41,48,49,52\}$\\
$N_{56,16}$&$\{2,8,12,17,27,38,40,46,51,54\}$\\
$N_{56,17}$&$\{5,11,30,37,38,39,40,42,45,46\}$\\
$N_{56,18}$&$\{3,4,5,17,23,29,31,33,41,49\}$\\
$N_{56,19}$&$\{5,10,14,20,22,28,33,37,43,55\}$\\
$N_{56,20}$&$\{5,16,17,19,20,38,43,45,46,56\}$\\
 \noalign{\hrule height0.8pt}
\end{tabular}
}
\end{center}
\end{table}

\subsection{Table $N_{10}(n)$}

Currently, it is not known whether there is an 
$s$-extremal singly even self-dual $[n,n/2,10]$ codes 
for $n=60,62,64,66,68,70$.
We summarize in Table~\ref{Tab:d10} the number $N_{10}(n)$
of known inequivalent 
$s$-extremal singly even self-dual
$[n,n/2,10]$ codes, along with the references.

\begin{table}[thb]
\caption{Minimum weight $10$}
\label{Tab:d10}
\begin{center}
{\small
\begin{tabular}{c|c|c||c|c|c}
\noalign{\hrule height0.8pt}
$n$ & $N_{10}(n)$ & References &
$n$ & $N_{10}(n)$ & References \\
\hline
46& $1$       &\cite{HMT}  &56& $\ge 92$ & Proposition~\ref{prop:56}\\
48& $\ge 322$ &\cite{BYK}, \cite{C-S}, \cite{Venkov}&58& $\ge 3$   &\cite{HGK}\\
50& $\ge 1507$&\cite{HM02}                          &60& ? & \\
52& $\ge 460$ & \cite{HT} (see \cite[p.~30]{BG})
	 &$\vdots$&$\vdots$\\
54& $\ge 9115$&\cite{YL}   &70& ? & \\
 \noalign{\hrule height0.8pt}
\end{tabular}
}
\end{center}
\end{table}

There are $11$ inequivalent pure double circulant
singly even self-dual $[58,29,10]$ codes and
these codes are denoted by $P_{58,i}$ $(i=1,\ldots,11)$~\cite[Table~2]{HGK}.
Note that the codes $P_{58,i}$ $(i=1,2,3)$ are $s$-extremal.
We verified that $P_{58,i}$ $(i=1,2,3)$ has no $s$-extremal
singly even self-dual $[58,29,10]$ neighbor, which are
inequivalent to $P_{58,i}$ $(i=1,2,3)$.

We examine the construction of 
$s$-extremal singly even self-dual $[60,30,10]$ codes.
By Proposition~\ref{prop:sec3},
if there is an $s$-extremal singly even self-dual $[60,30,10]$ code $D$,
then the neighbors $N_1(D)$ and $N_3(D)$ of $D$
are $s$-extremal and extremal singly even self-dual $[60,30,12]$ codes.
If there is an $s$-extremal and extremal singly even self-dual $[60,30,12]$ code
$C$ such that $C_1$ contains no vector of weight $10$,
then the neighbor $N_3(C)$ of $C$ is 
an $s$-extremal singly even self-dual $[60,30,10]$ code.
In this case, by~\cite[Theorem~5]{C-S},
the weight enumerators $W_1$ and $W_3$
of $C_1$ and $C_3$ are uniquely determined as follows
\begin{align*}
W_1=&
 33600y^{14} + 1717760y^{18} + 26376960y^{22} + 130152960y^{26}
\\&
 + 220308352y^{30} + \cdots,
\\
W_3=&
396y^{10} + 29640y^{14} + 1735580y^{18} + 26329440y^{22} +
130236120y^{26}
\\& + 
220208560y^{30} + \cdots.
\end{align*}
Currently, $13$ inequivalent 
$s$-extremal and extremal singly even self-dual $[60,30,12]$ codes are known
(see Table~\ref{Tab:d12}).
We verified that both $C_1$ and $C_3$ contain a vector of
weight $10$ for each $C$ of the $13$ codes.

\section{Existence of $s$-extremal singly even self-dual codes with
minimum weight 12}\label{sec:d12}

In this section, 
new $s$-extremal and extremal singly even self-dual codes with minimum weight $12$
are constructed for lengths $62$ and $64$.
The current knowledge on the existence of 
$s$-extremal singly even self-dual codes with minimum weight
$12$ is also presented.

\subsection{Range of lengths}

By~\cite[Proposition~3.1]{Kim},
if there is an $s$-extremal
singly even self-dual $[n,n/2,12]$ code, then $n \le 68$.
In addition,
if there is a singly even self-dual  $[n,n/2,12]$ code with $n \le 68$, 
then $n \in \{56,60,62,64,66,68\}$~\cite{C-S}.

\subsection{Length 60}
Three inequivalent
$s$-extremal and extremal singly even self-dual $[60,30,12]$ 
codes are known~\cite{DH62} and~\cite{TJ98}.
Recently, $13$ inequivalent
$s$-extremal and extremal singly even self-dual $[60,30,12]$ 
codes have been found~\cite{H18}.
We verified that three codes of the $13$ codes are
equivalent to the codes in~\cite{DH62} and~\cite{TJ98}.


\subsection{Length 62} 
Eight inequivalent
$s$-extremal and extremal singly even self-dual $[62,31,12]$ codes are 
known~\cite{DH62}.
One more $s$-extremal and extremal singly even self-dual $[62,31,12]$ 
code was constructed~\cite{RY}.

Let $C$ be a singly even self-dual code of length $n$.
Let $T$ be a coset of $C$, say $T=t+C$, where $t \not\in C$.
Let $C^0$ denote the subcode of $C$
consisting of all codewords which are orthogonal to $t$.
Then there are cosets
$C^1,C^2,C^3$ of $C^0$ such that ${C^0}^\perp = C^0 \cup C^1 \cup
C^2 \cup C^3$, where $C = C^0  \cup C^2$ and $T = C^1 \cup C^3$.
If the weight of $t$ is odd, then
\[
C^+(t)= (0,0,C^0) \cup  (1,1,C^2) \cup  (0,1,C^1) \cup  (1,0,C^3)
\]
is a self-dual code of length $n+2$~\cite{Tsai92}.
We found an $s$-extremal and extremal singly even self-dual $[62,31,12]$ 
code by the above construction.
Consider the extremal self-dual $[60,30,12]$ code
$C'_{62,6}$ in~\cite{DH62} and
the vector $t$ of length $60$ having the following support
\[
\{ 1, 3, 7, 9, 31, 37, 38, 48, 50, 53, 58 \}.
\]
We verified that the code  $(C'_{62,6})^+(t)$ is an
$s$-extremal and extremal singly even self-dual $[62,31,12]$ code
with automorphism group of order $5$.
Since none of the known nine codes has automorphism group
of order $5$, we have the following:

\begin{prop}\label{prop:62}
There are at least $10$ inequivalent
$s$-extremal and extremal singly even self-dual $[62,31,12]$ codes.
\end{prop}

\subsection{Length 64} 
Two inequivalent $s$-extremal and extremal singly even self-dual $[64,32,12]$ codes
were constructed in~\cite{CHK}.
We denote by $C_{64,1}$ and $C_{64,2}$
the $s$-extremal and extremal singly even self-dual $[64,32,12]$ codes
listed in~\cite[Tables~7 and 8]{CHK}, respectively.
Recently, $22$ more $s$-extremal and extremal singly even self-dual $[64,32,12]$ codes
have been constructed in~\cite{AHY}.
We found one more $s$-extremal and extremal singly even self-dual $[64,32,12]$ code.
The code $F_1$ in~\cite{CHK} is an extremal singly even self-dual
$[64,32,12]$ code and its generator matrix is listed 
in~\cite[Table~3]{CHK}.
We define the code $N_{64}$ as follows
\[
N_{64}=
\langle (F_1 \cap \langle x_1,x_2 \rangle^\perp), x_1,x_2 \rangle,
\]
where 
\begin{align*}
\supp(x_1)=&
\{ 23, 31, 33, 36, 39, 42, 44, 46, 49, 50, 55, 57, 59, \ldots, 64 \},
\\
\supp(x_2)=&
\{ 34, 35, 37, 39, 40, 41, 42, 43, 44, 47, 49, 51, 53, 54 \}.
\end{align*}
We verified that $N_{64}$ is an 
$s$-extremal and extremal singly even self-dual $[64,32,12]$ code,
which is equivalent to none of the above $24$ codes.
Therefore, we have the following:

\begin{prop}\label{prop:64}
There are at least $25$ inequivalent
$s$-extremal and extremal singly even self-dual $[64,32,12]$ codes.
\end{prop}

\subsection{Lengths 66 and 68} 
There are three inequivalent pure double circulant
extremal singly even self-dual $[66,33,12]$
codes and these codes are denoted by
$C_{66,i}$ $(i=1,21,25)$~\cite[Table~3]{GH}.
The two codes $C_{66,i}$ $(i=1,21)$ are
$s$-extremal and extremal singly even self-dual $[66,33,12]$ codes, while
$C_{66,21}$ is equivalent to D16 in~\cite[Table~III]{C-S}.

Recently, six inequivalent $s$-extremal and extremal singly even self-dual
$[68,34,12]$ codes have been found in~\cite{YIL},
under the assumption that they possess an automorphism of order $7$.

\subsection{Table $N_{12}(n)$}

If $n \in \{60,62,64,66,68\}$, then
there is an $s$-extremal singly even self-dual
$[n,n/2,12]$ code.
We summarize in Table~\ref{Tab:d12} the number $N_{12}(n)$
of known inequivalent 
$s$-extremal singly even self-dual $[n,n/2,12]$ codes,
along with the references.
%

\begin{table}[thb]
\caption{Minimum weight $12$}
\label{Tab:d12}
\begin{center}
{\small
\begin{tabular}{c|c|c||c|c|c}
\noalign{\hrule height0.8pt}
$n$ & $N_{12}(n)$ & References &
$n$ & $N_{12}(n)$ & References \\
\hline
$56$& ?        & &$64$& $\ge 25$& Proposition~\ref{prop:64}\\
$60$& $\ge 13$ & \cite{DH62}, \cite{H18}, \cite{TJ98}
	 &$66$& $\ge 2$ &\cite{C-S}, \cite{GH} \\
$62$& $\ge 10$ & Proposition~\ref{prop:62}&$68$& $\ge 6$ &\cite{YIL}\\
\noalign{\hrule height0.8pt}
\end{tabular}
}
\end{center}
\end{table}


\section{Existence of $s$-extremal singly even self-dual codes with
minimum weight 14}\label{sec:d14}


In this section,
many new $s$-extremal singly even self-dual codes with minimum weight $14$
are constructed for length $78$.
Note that the largest minimum weight among singly even
self-dual codes of length $78$ is $14$ (see \cite{DGH}).
The current knowledge on the existence of 
$s$-extremal singly even self-dual codes with minimum weight
$14$ is also presented.

\subsection{Range of lengths}
If there is
an $s$-extremal  singly even self-dual
$[n,n/2,14]$ code, then $70 \le n \le 94$~\cite[Corollary~II.4]{HK08}.
Currently,
$s$-extremal singly even self-dual codes with minimum weight $14$
are known for only lengths $76$ and $78$ (see~\cite[p.~30]{BG}).

\subsection{Length 76}
An $s$-extremal and optimal singly even self-dual $[76,38,14]$ code having
automorphism of order $19$ was constructed in~\cite{BY}.
This code is denoted by $C_{76}$ in~\cite{BY}.
By Proposition~\ref{prop:sec3}, two 
$s$-extremal and optimal singly even self-dual $[76,38,14]$ codes 
$N_1(C_{76})$ and $N_3(C_{76})$ are constructed from $C_{76}$.
We verified that 
$C_{76}$, $N_1(C_{76})$ and $N_3(C_{76})$ are inequivalent, and
$N_1(C_{76})$ and $N_3(C_{76})$ have automorphisms of order $19$.
It was shown in~\cite{DY} that there are three inequivalent
$s$-extremal and optimal singly even self-dual $[76,38,14]$ codes
with an automorphism of order $19$.
Hence, $N_i(C_{76})$ $(i=1,3)$ is equivalent to one of 
the codes in~\cite{DY}.

%

\subsection{Length 78}
By Lemma 2.4 in~\cite{BG},
a singly even self-dual code constructed from an extremal
doubly even self-dual $[80,40,16]$ code by subtracting
is an $s$-extremal and optimal singly even self-dual $[78,39,14]$ code.
There are ten inequivalent extremal double circulant doubly even self-dual
$[80,40,16]$ codes~\cite{GH2}.
These codes are denoted by $P_{80,i}$ $(i=1,2,\ldots,6)$
and $B_{80,i}$ $(i=1,2,3,4)$.
It was shown in~\cite{DH80} that there are $11$ extremal
doubly even self-dual $[80,40,16]$ codes with an automorphism of
order 19, up to equivalence.  
The $11$ codes are denoted by $C_{80,1}, C_{80,2},\ldots,C_{80,11}$.
We verified that all 
$s$-extremal and optimal singly even self-dual $[78,39,14]$ codes
constructed from the $21$ codes by subtracting the pairs of
coordinates are divided into $1942$ equivalent classes.
The $1942$ inequivalent codes are constructed from
the codes $C$ by subtracting the pairs $(i,j)$ of coordinates,
where the sets of the pairs are given by $S_k$ $(k=1,2,\ldots,8)$ 
as follows
\begin{center}
\begin{tabular}{c|c||c|c}
$S_k$ & $C$ & $S_k$ & $C$ \\
\hline
$S_1$ & $P_{80,1}$               &$S_5$ & $C_{80,i}$ $(i=1, 3, 6, 9)$ \\
$S_2$ & $P_{80,i}$ $(i=2,3,4,5)$ &$S_6$ & $C_{80,i}$ $(i=2, 4, 8, 10, 11)$ \\
$S_3$ & $P_{80,6}$               &$S_7$ & $C_{80,5}$\\ 
$S_4$ & $B_{80,i}$ $(i=1,2,3,4)$ &$S_8$ & $C_{80,7}$ \\
\end{tabular}
\end{center}
The sets $S_k$ $(k=1,2,\ldots,8)$ 
can be obtained from\\
\url{http://www.math.is.tohoku.ac.jp/~mharada/Paper/78-39-14-se.txt}.

\begin{prop}\label{prop:78}
There are at least $1942$ inequivalent 
$s$-extremal and optimal singly even self-dual $[78,39,14]$ codes.
\end{prop}

\subsection{Table $N_{14}(n)$}

We summarize in Table~\ref{Tab:d12} the number $N_{14}(n)$
of known inequivalent 
$s$-extremal singly even self-dual $[n,n/2,14]$ codes,
along with the references.

\begin{table}[thb]
\caption{Minimum weight $14$}
\label{Tab:d14}
\begin{center}
{\small
\begin{tabular}{c|c|c||c|c|c}
\noalign{\hrule height0.8pt}
$n$ & $N_{14}(n)$ & References &
$n$ & $N_{14}(n)$ & References \\
\hline
70 & ?        & &78 & $\ge 1942$    & Proposition~\ref{prop:78} \\
72 & ?        & &80 & ? &\\
74 & ?        & &$\vdots$ & $\vdots$ &\\
76 & $\ge 3$  &\cite{BY}, \cite{DY} &94 & ? &\\
\noalign{\hrule height0.8pt}
\end{tabular}
}
\end{center}
\end{table}

It was shown in~\cite{HM} that there is
an $s$-extremal singly even self-dual $[24k+8,12k+4,4k+2]$ code
if and only if
there is an extremal doubly even self-dual
$[24k+8,12k+4,4k+4]$ code
with covering radius $4k+2$.
The covering radii of some extremal doubly even self-dual codes of
length $80$ were determined in~\cite{HM}.
It is an open problem to determine whether there is
an extremal doubly even self-dual $[80,40,16]$ code with
covering radius $14$.

\section{Remarks on $s$-extremal singly even self-dual codes with
minimum weights at least 16}
For minimum weight $d \ge 16$, currently only two
$s$-extremal singly even self-dual codes are known.
More precisely, an $s$-extremal and extremal singly even self-dual
$[86,43,16]$ code and an $s$-extremal singly even self-dual
$[102,51,18]$ code are known (see~\cite[p.~30]{BG}).

\bigskip
\noindent
{\bf Acknowledgments.}
This work is supported by JSPS KAKENHI Grant Number 19H01802.
The author would like to thank Akihiro Munemasa
and Vladimir D. Tonchev for useful comments.




\begin{thebibliography}{99}

\bibitem{AGKSS}
C. Aguilar-Melchor, P. Gaborit, J.-L. Kim, L. Sok and P. Sol\'e,
Classification of extremal and $s$-extremal binary self-dual codes
of length 38,
{\sl IEEE\ Trans.\ Inform.\ Theory}
{\bf 58} (2012), 2253--2262.

\bibitem{AHY}D. Anev, M. Harada and N. Yankov,
New extremal singly even self-dual codes of lengths $64$ and $66$,
{\sl J.\ Algebra Comb.\ Discrete Struct.\ Appl.}
{\bf 5} (2018), 143--151.
	
\bibitem{BY} A. Baartmans and V. Yorgov, 
Some new extremal codes of lengths $76$ and $78$, 
{\sl IEEE Trans.\ Inform.\ Theory}
{\bf 49} (2003), 1353--1354.

\bibitem{BG}C. Bachoc and P. Gaborit, 
Designs and self-dual codes with long shadows,
{\sl J. Combin.\ Theory Ser.~A}
{\bf 105} (2004), 15--34.

\bibitem{B06}R.T. Bilous,
Enumeration of the binary self-dual codes of length 34,
{\sl J. Combin.\ Math.\ Combin.\ Comput.}
{\bf 59} (2006), 173--211.

\bibitem{Magma}W. Bosma, J. Cannon and C. Playoust, 
The Magma algebra system I: The user language, 
{\sl J. Symbolic Comput.}
{\bf 24} (1997), 235--265.


\bibitem{BBH}S. Bouyuklieva, I. Bouyukliev and M. Harada,
Some extremal self-dual codes and unimodular lattices in 
dimension $40$, 
{\sl Finite Fields Appl.}
{\bf 21} (2013), 67--83.

\bibitem{BYK}S. Bouyuklieva, N. Yankov and J.-L. Kim,
Classification of binary self-dual $[48, 24, 10]$ codes with
an automorphism of odd prime order,
{\sl Finite Fields Appl.}
{\bf 18} (2012) 1104--1113.

\bibitem{BP} R.~Brualdi and V.~Pless,
Weight enumerators of self-dual codes,
{\sl IEEE Trans.\ Inform.\ Theory}
{\bf 37} (1991), 1222--1225.

\bibitem{CHK}N. Chigira, M. Harada and M. Kitazume, 
Extremal self-dual codes of length $64$ through neighbors 
and covering radii,
{\sl Des.\ Codes Cryptogr.}
{\bf 42} (2007),  93--101.

\bibitem{C-S} J.H.~Conway and N.J.A.~Sloane,
A new upper bound on the minimal distance of self-dual codes,
{\sl IEEE\ Trans.\ Inform.\ Theory}
{\bf 36} (1990), 1319--1333.

\bibitem{DH62}R. Dontcheva and M. Harada, 
New extremal self-dual codes of length 62 and related extremal self-dual
codes,
{\sl IEEE Trans.\ Inform.\ Theory}
{\bf 48} (2002), 2060--2064. 

\bibitem{DH80}R. Dontcheva and M. Harada, 
Extremal doubly-even $[80,40,16]$ codes with an automorphism of 
order $19$,
{\sl Finite Fields Appl.}
{\bf 9} (2003), 157--167.

\bibitem{DY} R. Dontcheva and V. Yorgov,
The extremal codes of lengths $76$ with an automorphism of order 19,
{\sl Finite Fields Appl.} {\bf 9} (2003), 395--399.


\bibitem{DGH}S.T. Dougherty, T.A. Gulliver and M. Harada,
Extremal binary self-dual codes,
{\sl IEEE\ Trans.\ Inform.\ Theory}
{\bf 43} (1997), 2036--2047.


\bibitem{Elkies}N.D. Elkies, 
Lattices and codes with long shadows,
{\sl Math.\ Res.\ Lett.}
{\bf 2} (1995), 643--651.

\bibitem{G07}P. Gaborit, 
A bound for certain $s$-extremal lattices and codes,
{\sl Arch.\ Math.\ (Basel)}
{\bf 89} (2007), 143--151.


\bibitem{GH} T.A. Gulliver and M. Harada,
Classification of extremal double circulant self-dual codes of 
lengths $64$ to $72$,
{\sl Des.\ Codes Cryptogr.}
{\bf 13}  (1998),  257--269.

\bibitem{GH2} T.A. Gulliver and M. Harada,
{Classification of extremal double circulant self-dual codes of
lengths $74$--$88$},
{\sl Discrete Math.}
{\bf 306} (2006), 2064--2072.

\bibitem{HK08}S. Han and J.-L. Kim, 
Upper bounds for the lengths of $s$-extremal codes over 
$\FF_2$, $\FF_4$, and $\FF_2+u\FF_2$,
{\sl IEEE Trans.\ Inform.\ Theory}
{\bf 54} (2008), 418--422.

\bibitem{H18} M. Harada, 
Binary extremal self-dual codes of length $60$ and related codes,
{\sl Des.\ Codes Cryptogr.}
{\bf 86} (2018), 1085--1094. 

\bibitem{HGK} M. Harada, T.A. Gulliver and  H. Kaneta,
{Classification of extremal double-circulant self-dual codes of
length up to $62$},
{\sl Discrete Math.}
{\bf 188} (1998), 127--136.

\bibitem{HKWY}M. Harada, M. Kiermaier, A. Wassermann and R. Yorgova, 
New binary singly even self-dual codes,
{\sl IEEE Trans.\ Inform.\ Theory}
{\bf 56} (2010), 1612--1617.

\bibitem{Venkov} M. Harada, M. Kitazume, A. Munemasa and B. Venkov,
On some self-dual codes and unimodular lattices in dimension $48$,
{\sl European J. Combin.}
{\bf 26} (2005), 543--557.

\bibitem{HM02} M. Harada and A. Munemasa,
A quasi-symmetric $2$-$(49,9,6)$ design,
{\sl J. Combin.\ Des.}
{\bf 10 }(2002), 173--179.

\bibitem{HM36} M. Harada and A. Munemasa,
Classification of self-dual codes of length $36$,
{\sl Advances Math.\ Communications}
{\bf 6} (2012), 229--235.

\bibitem{HM}M. Harada and A. Munemasa,
On $s$-extremal singly even self-dual $[24k+8,12k+4,4k+2]$ codes,
{\sl Finite Fields Appl.}
{\bf 48} (2017), 306--317.

\bibitem{HMT}M. Harada, A. Munemasa and V.D. Tonchev, 
A characterization of designs related to an extremal 
doubly-even self-dual code of length $48$,
{\sl Ann.\ Comb.}
{\bf 9} (2005), 189--198.

\bibitem{HT}W.C. Huffman and V.D. Tonchev,
The $[52,26,10]$  binary self-dual codes with an automorphism of order $7$,
{\sl Finite Fields Appl.} {\bf 7}  (2001),  341--349.

 \bibitem{Kim}J.-L. Kim, 
Remarks on $s$-extremal codes,
Advances in coding theory and cryptography, 101--113, 
Ser.\ Coding Theory Crypto., 3, World Sci. Publ., Hackensack, NJ, 2007. 

\bibitem{MS73} C.L.~Mallows and N.J.A.~Sloane,
{An upper bound for self-dual codes}, 
{\sl Inform.\ Control}
{\bf 22} (1973), 188--200.

\bibitem{MG08}C.A. Melchor and P. Gaborit,
On the classification of extremal $[36,18,8]$ binary self-dual codes,
{\sl IEEE\ Trans.\ Inform.\ Theory} 
{\bf 54} (2008),  4743--4750.

	
\bibitem{Rains} E.M.~Rains,
{Shadow bounds for self-dual codes},
{\sl IEEE Trans.\ Inform.\ Theory}
{\bf 44} (1998), 134--139.

\bibitem{RY} R. Russeva and N. Yankov, 
On binary self-dual codes of lengths $60, 62, 64$ and $66$ 
having an automorphism of order $9$,
{\sl Des.\ Codes Cryptogr.}
{\bf 45} (2007), 335--346.

\bibitem{Tsai92} H.-P.~Tsai,
Existence of some extremal self-dual codes,
{\sl IEEE\ Trans.\ Inform.\ Theory}
{\bf 38}  (1992),  1829--1833.

\bibitem{TJ98} H.-P.~Tsai and Y.J.~Jiang,
{Some new extremal self-dual $[58,29,10]$ codes,}
{\sl IEEE\ Trans.\ Inform.\ Theory}
{\bf 44} (1998), 813--814.
	
\bibitem{YIL}N. Yankov, M. Ivanova and M.H. Lee,
Self-dual codes with an automorphism of order $7$ and
$s$-extremal codes of length $68$,
{\sl Finite Fields Appl.}
{\bf 51} (2018), 17--30. 

	
\bibitem{YL} N. Yankov and M.H Lee, 
New binary self-dual codes of lengths $50$--$60$,
{\sl Des.\ Codes Cryptogr.}
{\bf 73} (2014), 983--996.

\bibitem{Zhang}S. Zhang, 
On the nonexistence of extremal self-dual codes,
{\sl Discrete Appl.\ Math.}
{\bf 91} (1999), 277--286.

\end{thebibliography}
\end{document}